\documentclass[12pt,a4paper]{article}
\usepackage[active]{srcltx}
\usepackage{amsthm}
\usepackage{amssymb}
\usepackage{psfrag}
\usepackage{graphicx}
\usepackage{tikz}
\usetikzlibrary{arrows,calc,intersections,plotmarks,decorations.markings}
\newtheorem{theorem}{Theorem}
\newtheorem*{lemma}{Lemma}

\theoremstyle{definition}
\newtheorem{definition}{Definition}

\newtheorem{example}{Example}
\newtheorem{proposition}{Proposition}
\theoremstyle{remark}
\newtheorem{remark}{Remark}
\def\mes{\mathrm{mes}\,}
\def\col{\mathrm{col}\,}
\def\supp{\mathrm{supp}\,}

\def\sp{\mathrm{sp}\,}
\title{ Deterministic Diffusion}
\author{ L.~P.~Nizhnik, I.~L.~Nizhnik \\
Institute of Mathematics NAS of Ukraine, Kiev, Ukraine\\
nizhnik@imath.kiev.ua,\,irene@imath.kiev.ua}
\date{}

\begin{document}

 \maketitle

\begin{abstract}
In the present paper, we give a series of definitions and
properties of Lifting Dynamical Systems (LDS) corresponding to the
notion of deterministic diffusion. We present heuristic
explanations of the mechanism of formation of  deterministic
diffusion in LDS and the anomalous deterministic diffusion in the
case of transportation in long billiard channels with spatially
periodic structures. The expressions for the coefficient of
deterministic diffusion are obtained.
\end{abstract}


\noindent

\section{Introduction}\label{Sec:1}

 \smallskip
One-dimensional dynamical systems on the entire axis with discrete
time are defined by the recurrence relation
\begin{equation}\label{eq:1}
x_{n+1}=f(x_{n}),
\end{equation}
where  $f(x)$ is a real function given on the entire axis and
$x_0$ is a given  initial value \cite{Ka,ShK}. Equation
(\ref{eq:1}) determines a trajectory $(x_0,x_1,...,x_n,...)=x$ in
the dynamical system (\ref{eq:1}) according to the initial value
$x_0$ and the form of the function $f.$ For the dynamical systems
admitting the chaotic behavior of trajectories, the problem of
construction of the entire trajectory or even of determination of
the values of $x_n$ for large $n$ is quite complicated because, as
a rule, the numerical calculations are performed with a certain
accuracy and the dependence of the subsequent values of $x_n$  on
the variations of the previous values is unstable. Moreover, from
the physical point of view, the initial value $x_0$ is specified
with a certain accuracy. Therefore for the investigation of the
behavior of trajectories for large values of time, we can analyze
not  the evolution of system (\ref{eq:1}) but the evolution of
measures on the axis generated by  this evolution.

If a probability measure $\mu_{0}$ (a normalized measure for which
the measure of the entire axis is equal to 1) with density
$\rho_{0}:$\,$\mu_{0}(A)=\int\limits_{A}\,\rho_{0}(x)\,dx$ is
given at the initial time, then, for a  unit of time, system
(\ref{eq:1}) maps this measure into
$\mu_{1}:$\,$\mu_{1}(A)=\mu(f^{-1}(A)),$ where  $f^{-1}(A)$ is a
complete preimage of the set  $A$ under the map $f.$ The operator
mapping the measure $\mu_{0}$ into the measure $\mu_{1}$ is called
a Perron--Frobenius operator. The Perron--Frobenius operator
$\mathcal{F}$ is linear even for nonlinear dynamical systems
(\ref{eq:1}) and maps the density $\rho_{0}$ of the initial
measure into the density  $\rho_{1}$ of the measure $\mu_{1}$ as
the integral operator with singular kernel containing a Dirac
$\delta$ function:
 \begin{equation}\label{eq:2}
\rho_{1}(x)=\mathcal{F}\rho_{0}(x)=\int\,\delta(x-f(y))
\rho_0(y)\,dy=\sum_{y_k \in f^{-1}(\{x\}) }\,
\frac{1}{|f'(y_k)|}\rho_0(y_k).
\end{equation}

The investigation of the asymptotic behavior of the density
$\rho_{n}=\mathcal{F}^{n}\rho_0$ as  $n\rightarrow\infty$ is
reduced to the investigation of the behavior of the semigroup
$\mathcal{F}^{n}.$ There are examples of dynamical systems
(\ref{eq:1}) with locally stretching maps $f$ for which the
densities $\rho_{n}$ are asymptotically Gaussian as
$n\rightarrow\infty$ independently of the choice of the density of
the initial probability measure. In this case, it is said that
deterministic diffusion occurs in the dynamical system
(\ref{eq:1}).

The aim of the present paper is to consider examples of dynamical
systems with deterministic diffusion. We restrict ourselves to the
so-called lifting dynamical systems (LDS) with piecewise linear
functions $f(x)$ in DS (\ref{eq:1}).

We briefly consider the mechanisms of  appearance of the anomalous
deterministic diffusion in the process of transportation in long
billiard channels with spatially periodic structures.

\section{Lifting Dynamical System}

Consider the dynamical system (\ref{eq:1}) on the entire axis,
where the function $f(x)$ is given on the main interval
$I_0=[-\frac{1}{2},\frac{1}{2}).$ The function $s(x)=f(x)-x$ with
$x\in I_0$ has the sense of a shift of the point $x$ under the map
$f.$ We split the entire axis $(-\infty,\infty)$ into disjoint
intervals $I_k=[k-\frac{1}{2},k+\frac{1}{2}),$ where $k\in Z$  are
integers. Assume that the function  $f$ is extended from the
interval $I_0$ onto all intervals $I_k$ so that the shift of a
point under the map $f$ on each interval $k$ is the same as on
$I_0,$ that is $s(k+x)=f(k+x)-(k+x)=s(x)=f(x)-x.$ This gives a
periodic shift function  $s(x)$ and implies the periodicity of the
function $f$ with lift 1
 \begin{equation}\label{eq:3}
f(k+x)=k+f(x),\, |x|<\frac{1}{2}, \,\, k \in Z.
\end{equation}

The dynamical system (\ref{eq:1}) with a function  $f$  satisfying
property (\ref{eq:3}) is called a LDS. This dynamical system (DS)
is well known and thoroughly investigated  \cite{CAM,Ka,Kl}.

\begin{lemma}\label{L:1}
Let a function $f$ defined on the interval
$I_0=[-\frac{1}{2},\frac{1}{2})$ be a piecewise monotone
stretching function on a finite partition $\{ I_{0,j}
\}_{j=1}^{m}$ of the interval $ I_{0}=\bigcup\limits_{j=1}^{m}
I_{0,j},$ i.e., there exists  $\lambda > 1$ such that
\begin{equation}\label{eq:4}
|f(x)-f(y)|\geq \lambda |x-y|
\end{equation}
for  $x, y \in I_{0,k},$ and, in addition, the function $f(x)$ has
finite nonzero discontinuities at the points of joint of the
intervals  $ I_{0,j}.$ Then the periodic extension (\ref{eq:3})
with lift 1 specifies a locally stretching function $f$ on the
entire axis.
\end{lemma}
 \medbreak
\begin{proof}
If the points $x$ and $y$ belong to the same interval  $I_{0,k},$
then, by virtue of (\ref{eq:4}),
$$
|f(x)-f(y)|=|f(x-k)-f(y-k)|\geq \lambda |x-y|.
$$

If the points $x$ and $y$ belong to neighboring intervals, then
the function has a discontinuity at the point of joint of these
intervals and the value of this discontinuity is not smaller than
a certain $a > 0.$ In view of the monotonicity of the function
$f,$ there exists $\varepsilon_{0}>0$ such that, for
$|x-y|<\varepsilon_{0},$ the points $x$ and $y$ belong either to
the same interval $I_{0,k}$ or to neighboring intervals. Hence,
$$
|f(x)-f(y)|\geq\frac{a}{3}\geq \lambda \varepsilon,
$$
where $\varepsilon=\min (\varepsilon_{0},\frac{a}{3\lambda}).$
Thus, for $|x-y|<\varepsilon,$ we have $|f(x)-f(y)|\geq \lambda
|x-y|.$
\end{proof}

Let $x=(x_0,x_1,...,x_n,...)$ be a trajectory for the LDS
(\ref{eq:1})--(\ref{eq:3}). By $k,$ we denote the number of the
interval $I_k=[k-\frac{1}{2},k+\frac{1}{2})$ containing a number
$a \in I_k$ as follows: $k=[a)$ is the nearest integer for the
number  $a.$ Then the trajectory $x$ is associated with an
integer-valued sequence:
$M=([x_0),[x_1),...,[x_n),...)=(m_1,m_2,...,m_n,...)$  of the
numbers of intervals containing the values $x_n.$ The sequence $M$
is called the route of the trajectory $x$ and contains the numbers
of intervals successively ``visited'' by the phase point in the
LDS \cite{Ka}.

\begin{proposition}\label{p:1}
Let the function $f$ determining the LDS
(\ref{eq:1})--(\ref{eq:3}) be stretching. Then the trajectory of
the LDS is uniquely determined by its route.
\end{proposition}
\begin{proof}

Assume that two initial conditions $x_0^{(1)}$ and $x_0^{(2)}$
generate trajectories of the LDS with the same route. This means
that the points $x_n^{(1)}$ and $x_n^{(2)}$ lie in the same
interval $I_{m_n}.$ In view of the stretching property of the map $f,$
the inequality $|x_{n+1}^{(1)}-x_{n+}^{(2)}|\geq
\lambda|x_n^{(1)}-x_n^{(2)}|,$\, $\lambda > 1,$ is true. Since any
two points from the same interval $I_{m_n}$ differ by at most 1,
the successive application of this inequality yields the
inequality $|x_0^{(1)}-x_0^{(2)}|\leq \frac{1}{\lambda^{n}}$ for
any $n.$ Since $\lambda >  1,$ this implies that
$x_0^{(1)}=x_0^{(2)}.$
\end{proof}

\begin{definition}\label{D:1}
If any integer-valued sequence in the LDS is a route of a certain
trajectory and the trajectory is uniquely determined by its route,
then we say that the LDS possesses the Bernoulli property
\cite{Ka}.
\end{definition}

\begin{proposition}\label{p:2}
If a function $f$ defined in the interval
$I_0=[-\frac{1}{2},\frac{1}{2})$ is odd, continuous, monotonically
increasing, stretching, and unbounded, then the LDS
(\ref{eq:1})--(\ref{eq:3}) corresponding to this function $f$
possesses the Bernoulli property.
\end{proposition}

\begin{proof}
The uniqueness of reconstruction of a trajectory according to its
route is proved in Appendix~\ref{p:1}. Let
$M=(m_1,m_2,...,m_n,...)$ be an arbitrary integer-valued sequence.
In the interval $I_{m_1},$ we consider a sequence of embedded
contracting closed intervals  $U_n$ constructed as follows:

Let $U^{(1)}$ be the closure of the preimage of the interval
$I_{m_n}$ under the map $f$ lying on $I_{m_{n-1}}$ :
$U^{(1)}=f^{-1}(I_{m_n}).$ Since the  map $f$ is stretching, the
length $\mes U^{(1)} \leq \lambda^{-1}.$ Let $U^{(2)}$ be the
preimage of $U^{(1)}$ in the interval $I_{m_{n-2}}$ and let, by
induction, $U^{(k)}$ be the preimage of $U^{(k-1)}$ in the
interval $I_{m_{n-k}}.$ The sets $U_n=U^{(n)}\subset I_{m_1}$ and
$\mes U_n \leq \lambda^{n}.$ As $n\rightarrow \infty,$  we get a
system of embedded closed sets $U_{n+1}\subset U_n$ in the
interval $I_{m_{k}}$  and  $\mes U_n \rightarrow 0$ as
$n\rightarrow \infty.$ The common limit point $x_0$ for all $U_n$
lies in $I_{m_1}$ and its trajectory
$x=(x_0,f(x_0),...,f^{n}(x_0),...)$ has the route $M.$
\end{proof}

\noindent  \section{Markov Partition of the Phase Space of LDS.
Piecewise Linear LDS}

The system of intervals
$I_k=[k-\frac{1}{2},k+\frac{1}{2}),$\,$k\in Z,$ for  the LDS
(\ref{eq:1})--(\ref{eq:3}) can be regarded as a Markov partition
of the phase space \cite{Ka}. Consider a finer Markov
subpartition. Assume that the main interval $I_0$ is split into
finitely many $m$  subintervals $I_{0,j}=[x_j, x_{j+1}),$ where
$x_0=-\frac{1}{2},$\,$x_j
<x_{j+1},$\,$x_{m}=\frac{1}{2},$\,$j=0,1,2,...,m,$ and $I_{0,j}=[
x_{j-1},x_j).$ An integer-valued shift of this partition leads to
the decomposition of the intervals $I_{k}$ into subintervals
$I_{k,j}=[k+x_{j-1}, k+x_{j}).$ Hence, the entire axis is split
into intervals  $\{I_{k,j}\},$\, $k\in Z,$ \, $j= 1, 2,...,m.$

\begin{definition}\label{D:2}
We say that the Markov partition $\{I_{k,j}\}_{k\in Z,1\leq j \leq
m}$ of the entire axis is consistent with the LDS
(\ref{eq:1})--(\ref{eq:3}) if the LDS maps, for a unit of time, any probability measure
with constant densities in each set $I_{k,j}$ into a measure with
constant densities in each interval $I_{k,j}$.
\end{definition}

\begin{definition}\label{D:3}
We say that a function $f$ specifying the LDS
(\ref{eq:1})--(\ref{eq:3}) is consistent with the Markov partition
$\{I_{k,j}\}_{k\in Z,1 \leq j \leq m}$ defined with the help of
the numbers $-\frac{1}{2}=x_0,x_1,...,x_{m}=\frac{1}{2}$ if it is
linear and nonconstant on each interval $I_{k,j}$ and, at each end
of the interval $I_{k,j},$ takes values equal to an integer plus
one of the numbers $x_j,$\, $j=0,1,...,m.$
\end{definition}

\begin{example}\label{Ex:1}
Let $f(x)=\Lambda x$ be a linear function in the interval
$I_0=[-\frac{1}{2},\frac{1}{2}),$ where $\Lambda=2l+1$ is an odd
number. Then this function is consistent with the Markov partition
$\{I_{k}\}_{k\in Z},$\,$I_k=[k-\frac{1}{2},k+\frac{1}{2}).$
\end{example}

\begin{example}\label{Ex:2}
Let  $f(x)=\Lambda x,$ where $\Lambda=2l$ is an even number. Then
this function is consistent with the Markov partition
$\{I_{k,\pm}\}_{k\in Z},$ where $I_{k,+}=[k,k+\frac{1}{2})$ and
$I_{k,-}=[k-\frac{1}{2},k).$
\end{example}

An important example of the function $f$ satisfying
Definition~\ref{D:3} is given by the following assertion:

 \begin{proposition}\label{p:3}
Assume that a function  $f(x)$ determining the LDS in the interval
$I_0$ is piecewise linear, takes half-integer values at the ends
of each its linear pieces, and moreover, $f'(x)\neq 0$ almost
everywhere. Then, by Definition~\ref{D:3}, the function $f$ is
consistent with the Markov partition $\{I_k\}|_{k \in Z}.$
\end{proposition}

\begin{proof}
The proof follows from the verification of the conditions of
Definition~\ref{D:3} for points $\{x_j\},$ that is for  all points
at which the function $f$ takes half-integer values.
\end{proof}

A piecewise linear function $f$ from Proposition~\ref{p:3} will be
called a piecewise linear function  taking half-integer values at
the ends of the linear pieces.

The equivalence of Definitions~\ref{D:2} and \ref{D:3} for LDS
yields the following statement:

\begin{proposition}\label{p:4}
For the Markov partition $\{I_{n,j}\}_{n\in Z,|j|\leq m}$ to be
consistent with the action of the LDS (\ref{eq:1})--(\ref{eq:3})
by Definition~\ref{D:2}, it is necessary and sufficient that the
function $f$ determining the LDS (\ref{eq:1})--(\ref{eq:3}) be
consistent with this Markov partition by Definition~\ref{D:3}.
\end{proposition}

\begin{proof}
We now show that if $f$ satisfies the conditions of
Definition~\ref{D:3}, then the LDS transforms the probability
measure with constant densities on $I_{n,j}$ into a measure with
constant densities on $I_{n,k},$ i.e., the LDS satisfies the
conditions of Definition~\ref{D:2}. Since operator (\ref{eq:2}) is
linear, it suffices to consider the case where the density of the
initial measure is constant on the fixed interval
$I_{n_{0},k_{0}}.$ On this interval, the function $f(x)$ is linear
and nonconstant and maps $I_{n_{0},k_{0}}$ into the union $\bigcup
I_{n,k}$ of several neighboring intervals. Therefore, the inverse
map of these intervals is linear and, hence, the measure has a
constant density in each of these intervals.

Let the conditions of Definition~\ref{D:2} be satisfied. If the
density of the initial measure is constant in the interval
$I_{n_{0},k_{0}},$ then the map $f$ gives a new measure  $\mu_1$
on the entire axis. It is clear that $\supp \mu$ coincides with
the closure of a certain union $f^{-1}(I_{n,k})=\bigcup I_{0,j}$
because, otherwise, there exist an interval $I_{\bar{n},\bar{k}}$
and its part  $A$  such that
$\mu_1(I_{\bar{n},\bar{k}})=\mu_1(A)\neq 0$  and $\mu_1(A)=0.$
This contradicts the condition of Definition~\ref{D:2} according
to which the measure $\mu_1$ has a constant density. Let
$I_{\bar{n},\bar{k}}$ be one of the intervals $f(I_{n_{0},k_{0}})$
and let $B=f^{-1}(I_{\bar{n},\bar{k}}).$ Since $B\subset
I_{n_{0},k_{0}},$ the initial measure on $B$ has a constant
density, $f$ bijectively maps $B$ onto $I_{\bar{n},\bar{k}},$ and,
according to the Perron--Frobenius operator, the density
$\rho_1(x)$ of the measure  $\mu_1$ on $I_{\bar{n},\bar{k}}$ is
expressed via the density $\rho_0$ of the initial measure on
$I_{n_{0},k_{0}}$ by the equality
$\rho_1(x)=\frac{\rho_0(x)}{|f'(x)|}.$ Thus, we arrive at the
conclusion that $f'(x)\equiv const$ on the interval $B$ and that
the function maps the ends of the interval $B$ into the ends of
the interval $I_{\bar{n},\bar{k}}.$ Thus, the function $f$
satisfies the conditions of Definition~\ref{D:3}.
\end{proof}

The following question arises: Is it possible to construct Markov
partitions of the entire axis consistent with the linear function
$f(x)=\Lambda x$ for the values of $\Lambda$ other than in
Examples~\ref{Ex:1} and \ref{Ex:2}?

Since the linear function $f(x)=\Lambda x$ is odd, the numbers
$x_{j}$ specifying a consistent Markov subpartition of the
interval $I_0$ are symmetric about the middle of the interval
$I_0.$ Hence, it is sufficient to define solely the positive
values of $x_{j}$ and indicate that  the number $m$ of
subintervals $\{I_{0,j}\}_{1\leq j \leq m}$ is even or odd,
because $s=0$ belongs to the set $\{s_{j}\}$ for even $m$ and does
not belong to this set for odd $m.$

We enumerate the ends of the intervals $\{I_{k,j}\}$ located on
the positive half axis in the order of increase starting from
the first positive number. This yields a sequence $0
<s_{0}<s_{1}<s_{2}<...<s_{k}<....$ Since the Markov subpartition
$\{I_{k,j}\}$ is consistent with the LDS with the linear function
$f(x)=\Lambda x,$ we get
\begin{equation}\label{eq:3.1}
\Lambda s_{1}= s_{1+n_1},\, \Lambda s_{2}=\Lambda
s_{1+n_1+n_2},...,\Lambda
s_{\hat{m}}=s_{1+n_1+n_2+...+n_{\hat{m}}}.
\end{equation}
Here, $\hat{m}$ is expressed via $m$ of the form $m=2\hat{m}$ for
even $m$ and of the form $m=2\hat{m}-1$ for odd $m.$ Hence, each
Markov subpartition of the axis consistent with the linear
function $f(x)=\Lambda x$ specifying the LDS
(\ref{eq:1})--(\ref{eq:3}) is associated with two parameters: the
parity of the number $m$ of subintervals $\{I_{0,j}\}$ (i.e.,
$s=(-1)^{m}$) and the integer-valued vector
$\vec{n}=(n_{1},n_{2},...,n_{\hat{m}}).$ These quantities
$(s,\vec{n})$ are called the parameters of the Markov subpartition
$\{I_{k,j}\}.$

\begin{proposition}\label{p:3aa}
The parameters of the Markov subpartition $\{I_{k,j}\}$, i.e.,
the parity of the number  $m$ and the integer-valued vector
$\vec{n}=(n_{1},n_{2},...,n_{\hat{m}})$ uniquely define the
quantity $\Lambda,$ the slope of the linear function
$f(x)=\Lambda x,$ consistent with the Markov subpartition
$\{I_{k,j}\}$ and a collection of numbers
$\{s_{j}\}_{j=1}^{\hat{m}}$ specifying the ends of the intervals
$I_{k,j}=[k+s_{j},k+s_{j+1})$ of this subpartition.
\end{proposition}

\begin{proof}
Let the number of components of the vector $\vec{n}$ be equal to
$\hat{m}.$ Then the number $m$ of subintervals $\{I_{0,j}\}$ in
the Markov subpartition is given by the equality $m=2\hat{m}$ for
even $m$ and the equality $m=2\hat{m}-1$ for odd $m$ (the parity
of $m$ is specified by the quantity $s=(-1)^{m}$). We always have
$s_{\hat{m}}=\frac{1}{2}.$ For odd $m,$ the sequence $s_{1},
s_{2},...,s_{k},...$ has the form $s_{1},
s_{2},...,s_{\hat{m}+1},\frac{1}{2},1-s_{\hat{m}-1},1-s_{\hat{m}-2},...,1-s_{1},1+s_{1},1+s_{2},...,1+s_{\hat{m}-1},\frac{3}{2},2-s_{\hat{m}-1},...$
and is explicitly expressed via $s_{1},
s_{2},...,s_{\hat{m}}=\frac{1}{2}.$ Hence, we can explicitly
express each term of the sequence $s_{1}, s_{2},...,s_{k},...$ in
terms of an integer and one of the numbers  $s_{1},
s_{2},...,s_{\hat{m}}.$ Thus, equalities (\ref{eq:3.1}) turn into
a linear system for $s_{1}, s_{2},...,s_{\hat{m}}$ whose
coefficients are either integers or  the quantity
$\Lambda.$ The consistency condition for this system can be
formulated as the equality  of the determinant of this system to
zero. This gives the following algebraic equation for $\Lambda$:
$$
R_{\vec{n}}(\Lambda)=0.
$$

If we find  $\Lambda,$ then all $s_{1}, s_{2},...,s_{\hat{m}}$ can
be uniquely determined from (\ref{eq:3.1}). Similarly, we consider
the case of even $m.$ In this case, the sequence $s_{1},
s_{2},...,s_{k},...$ contains integer values and has the form
$s_{1},
s_{2},...,s_{\hat{m}}=\frac{1}{2},1-s_{\hat{m}-1},1-s_{\hat{m}-2},...,1-s_{1},0,1+s_{1},1+s_{2},....$
 \end{proof}

An implementation of this scheme by a constructive example is given in the next Proposition.

\begin{proposition}\label{p:6}
A linear function $f(x)=\Lambda x$ determining the LDS (\ref{eq:1})--(\ref{eq:3}) will be consistent with the Markov partition
$\{ [k-\frac12 ,\xi),[k-\xi ,k+\xi ),[k+\xi,k+\frac12 )\}_{k\in \mathbb Z}$, if for integers $0<m<n$ and the values $\varepsilon_1=\pm 1$, $\varepsilon_2=\pm 1$ we have
$$
\Lambda=\frac{2n+\varepsilon_2+\sqrt{(2n-\varepsilon_2)^2+8m\varepsilon_1}}2,
$$
$$
\xi =\frac{2m}{2n-\varepsilon_2+\sqrt{(2n-\varepsilon_2)^2+8m\varepsilon_1}}.
$$
\end{proposition}

\begin{proof} This is a consequence of the equalities (\ref{eq:3.1}), in this case having the form
$$
\Lambda \xi =m+\varepsilon_2\xi ,
$$
$$
\frac{\Lambda}2=n+\varepsilon_1\xi .
$$
 \end{proof}

It is worth noting that the set of values of the slope $\Lambda$
of the linear function $f(x)=\Lambda x,$\,$x \in
I_0=[-\frac{1}{2},\frac{1}{2}),$ for which it is possible to
construct Markov partitions consistent with the LDS
(\ref{eq:1})--(\ref{eq:3}) is everywhere dense on the half axis
$(2, \infty)$ \cite{CAM, Kl}.

\noindent  \section{Deterministic Diffusion}\label{Sec:4}

We consider the LDS (\ref{eq:1})--(\ref{eq:3}) for which the
collection of intervals $\{I_{k}\}_{k\in Z},$\,
$I_k=[k-\frac{1}{2},k+\frac{1}{2}),$\ forms a Markov partition of
the phase space consistent with the action of the LDS. This means
that, for a time unit, the LDS maps the measure $\mu_0$ with unit
density in the interval $I_0$  into a probability measure with
constant densities $p_{k}\geq 0$ in the intervals $I_k$ and
$\sum\limits_{k \in Z}\,p_k=1.$ As a result of multiple
application of the LDS, the initial measure $\mu_0$ is
transformed into a measure with constant densities in the
intervals $I_k.$ Let $P_k(n)$ be the density of the measure in the
interval $I_k$ after the $n$-fold action of the LDS upon the
initial measure $\mu_0.$ Then
\begin{equation}\label{eq:4.1}
P_k(n+1)=\sum\limits_{l \in Z}\,p_{k-l}P_l(n)
\end{equation}
and $P_k(0)=\delta_{k,0}$ due to the choice of the initial measure
$\mu_{0}.$ The asymptotic behavior of the quantities $P_k(n)$ for
large  $n,$ as solutions of Eq.~(\ref{eq:4.1}), is described by
the well-known central limit theorem of the probability
theory~\cite{F,GK}. Indeed, if we consider the sum
$\xi=\xi_{1}+\xi_{2}+...+\xi_{n}$ of $n$ independent random
variables each of which takes only integer values $k$ with
probability $p_{k},$ then we get Eq.~(\ref{eq:4.1}), where
$P_k(n)$ is the probability of the event that $\xi$ takes the
value $k.$

\begin{theorem}\label{Th:1}
Let the numbers $\{p_{k}\}_{k \in Z}$ take nonnegative values such
that $\sum\limits_{k \in Z}\,p_k=1,$  let the greatest common
divisor of the numbers $k$ with $p_{k}>0$ be equal to 1, and let
there exist the first and second moments

\begin{equation}\label{eq:4.2}
\sigma_1=\sum\limits_{k \in Z}\,k
p_k,\quad\sigma^{2}=\sum\limits_{k \in Z}\,k^{2} p_k.
\end{equation}

Then the solution of Eq.~(\ref{eq:4.1}) as $n\rightarrow \infty$
has an asymptotics
\begin{equation}\label{eq:4.3}
P_k(n)-\frac{1}{\sigma\sqrt{2\pi
n}}e^{-\frac{(x-\xi_n)^2}{2\sigma^2 n}}\rightarrow 0
\end{equation}
uniformly in  $k.$ The character of convergence in (\ref{eq:4.3})
depends on additional conditions, e.g., on the existence of the
third moment (the Laplace theorem).
\end{theorem}
\medbreak

\begin{proof}
We now briefly present the well-known scheme of the proof of
Theorem~\ref{Th:1} based on the fact that Eq.~(\ref{eq:4.1}) is a
difference analog of the convolution equation. As a result of the
Fourier transformation, the convolution turns into the product.
Let $P(\lambda, n)= \sum\limits_{k \in Z}\, P_k(n) e^{ik\lambda}$
be the characteristic function of the solution of
Eq.~(\ref{eq:4.1}), $P(\lambda)= \sum\limits_{k \in Z}\, p_k
e^{ik\lambda}.$ Thus, we get the following formula from
Eq.~(\ref{eq:4.1}):
\begin{equation}\label{eq:4.4}
P(\lambda,n)= [P(\lambda)]^n P(\lambda,0).
\end{equation}

Since the initial measure is concentrated on $I_0$ and its density
is constant, we conclude that $P(\lambda,0)=1$ and
\begin{equation}\label{eq:4.5}
P_k(n)=\frac{1}{2\pi} \int\limits_{-\pi}^{\pi} [P(\lambda)]^n
e^{-ik\lambda}\,d\lambda
\end{equation}

As $\lambda\rightarrow 0,$ $P(\lambda)= \sum\limits_{k \in
Z}\,p_k e^{ik\lambda}= 1+
i\sigma_1\lambda-\frac{\sigma^2\lambda^2}{2}+o(1)=e^{i\sigma_1\lambda-\frac{\sigma^2\lambda^2}{2}}+o(1).$
Hence,
$$
[P(\lambda)]^n=e^{i\sigma_1 n \lambda -\frac{\sigma^2  n
\lambda^2}{2}}+o(1).
$$
Substituting this result in (\ref{eq:4.5}) and integrating along
the entire axis, we obtain (\ref{eq:4.3}).

\end{proof}

The degree of closeness of two probability measures $\mu$ and
$\nu$ with densities $p(x)$ and $q(x)$  on the axis is often
estimated by the value $d(\mu,\nu)$ of the deviation (uniform in
$x)$ of the distribution functions
$$
P(x)=\int\limits_{-\infty}^{x}\,p(s)\,ds=\mu((-\infty,x)),
$$
$$
Q(x)=\int\limits_{-\infty}^{x}\,q(s)\,ds=\nu((-\infty,x)),
$$
i.e.,
\begin{equation}\label{eq:4.6}
d(\mu,\nu)= \sup\limits_{x}\,|P(x)-Q(x)|
\end{equation}

\begin{definition}\label{D:4}
We say that two sequences of measures $\mu_{n}$ and $\nu_{n}$ are
asymptotically equivalent as  $n\rightarrow\infty$ if
$\lim\limits_{n\rightarrow \infty}d(\mu_{n},\nu_{n})=0.$ For
normal measures $\nu_{n}$  with variances $\sigma^{2}_{n}$ and
means $\xi_{n},$ i.e., in the case where the density of measures
has the form of a Gauss curve
\begin{equation}\label{eq:4.7}
q_{n}(x)=\frac{1}{\sqrt{2\pi\sigma_{n}^{2}}}e^{-\frac{(x-\xi_{n})^2}{2\sigma_{n}^{2}}},
\end{equation}
we say that the sequence of measures  $\mu_{n}$ is asymptotically
normal as $n\rightarrow\infty.$ In addition, if
$\frac{1}{2}\sigma_{n}=D n,$ then we say that the sequence of
measures $\mu_{n}$ determines a normal diffusion with diffusion
coefficient $D.$ If the variance $\sigma_{n}$ depends nonlinearly
on $n,$ then the deterministic diffusion is anomalous.
\end{definition}

The result of Theorem~\ref{Th:1} for the LDS can be interpreted as
follows: The initial measure with unit density in the interval
$I_0$ can be regarded as randomly specified  initial data $x_{0}$
for the LDS uniformly distributed over the interval  $I_0.$ Then,
for large times, as $n \rightarrow \infty,$ the position of
$x_{n}$ is randomly distributed according to the normal law with
mean value  $\xi_n$ and variance $\sigma_n^{2}.$

In other words, in this case, we have the deterministic diffusion
with the diffusion coefficient $D=\frac{\sigma_n^{2}}{2n}$. Our
first aim is to study this phenomenon.

First, we reformulate the result of Theorem~\ref{Th:1} for the LDS
(\ref{eq:1})--(\ref{eq:3}) with a piecewise-linear function $f.$

\begin{theorem}\label{Th:2}
Assume that a function $f$ determining the LDS
(\ref{eq:1})--(\ref{eq:3}) is a piecewise linear function taking
different half-integer values at the ends of all linear parts.

Suppose that there exists
\begin{equation}\label{eq:4.6}
D=\frac{1}{2}\,\int\limits_{-\frac{1}{2}}^{\frac{1}{2}}\,|f(x)|^2\,dx-\frac{1}{24}.
\end{equation}

Then, after $n$ iterations in  the LDS, the initial measure
$\mu_{0}$ with unit density in the interval $I_0$ is
asymptotically mapped into a measure with normal distribution and
the diffusion coefficient $D.$
\end{theorem}
\medbreak
\begin{proof}
In view of Theorem~\ref{Th:1} and Propositions~\ref{p:3} and
\ref{p:4}, it is necessary to show that
\begin{equation}\label{eq:4.62}
\sigma^{2}= \sum\limits_{k \in Z}\,k^2 p_k=
\int\limits_{-\frac{1}{2}}^{\frac{1}{2}}\,|f(x)|^2\,dx-\frac{1}{12}.
\end{equation}

This can readily be proved because the integral of the piecewise
linear function $f$ can easily be taken. If the function $f(x)$ is
linear in the segment  $[a,b]$ and takes  values $k-\frac{1}{2}$
and  $k+\frac{1}{2}$ at the ends of this segment, then
$\int\limits_{a}^{b}\,|f(x)|^2\,dx=(k^{2}+\frac{1}{12})(b-a).$

We get $p_k=\mu_{1}(I_{k})= \mu(f^{-1}(I_{k})).$ The set
$f^{-1}(I_{k}) \bigcap I_{0}$ is formed by several subintervals
$I_{0,j}$ of the interval  $I_0$ in which the function  $f(x)$
takes values from the interval $[j-\frac{1}{2},j+\frac{1}{2}).$
Hence, $\int\limits_{f^{-1}(I_{k}) \bigcap
I_{0}}\,|f(x)|^2\,dx=(k^{2}+\frac{1}{12})\sum\mu_{0}(I_{0,j})=(k^{2}+\frac{1}{12})p_k.$

As a result of summation over  $k,$ we obtain
$$
\int\limits_{-\frac{1}{2}}^{\frac{1}{2}}\,|f(x)|^2\,dx=
\sum\limits_{k \in Z}\,(k^{2}+\frac{1}{12})p_k,
$$
which is equivalent to (\ref{eq:4.62}).
\end{proof}

\begin{example}\label{Ex:5}
  Let $f(x)=\Lambda x$, where $\Lambda$ is a positive
  number. Then function $f$
   satisfies the condition of Theorem~\ref{Th:2}
   if and only if $\Lambda $ is an odd number and $\Lambda > 1$.
  Formula~(\ref{eq:4.6}) gives $D=1/24(\Lambda^2-1)$, which of
  course, coincides with the formula known in the literature
  (see, for example,~(23.21) in~\cite{CAM} and references therein.)
\end{example}
\begin{example}\label{Ex:6}(zig-zag map)
  On the interval $[-1/2,\, 1/2 ],$ let us consider an odd piecewise
  linear function $f$ that takes the half-integer value $f(\xi)=p+1/2$
  at a point $\xi\, (0<\xi< 1/2)$. Let $f(0)=0$ and $f(1/2)=1/2.$ Then
  the diffusion coefficient, according to ~(\ref{eq:4.6}), has the
  following form:
  $$D=\frac{p+1}{12}(2p+1-2\xi).$$
  It argees with known results
  (see~\cite{KK} and the references therein.)
\end{example}

By using the results of Theorems~\ref{Th:1} and \ref{Th:2}, we can
give the following definition of  deterministic diffusion for the
dynamical system (\ref{eq:1}):
\begin{definition}\label{D:5}
We say that the one-dimensional DS (\ref{eq:1}) has a
deterministic diffusion if, for any initial probability measure
$\mu_{0}$ with bounded density, there exist a sequence of numbers
$\sigma_{n}^{2}> 0$ and $\xi_{n}$ such that the sequence of
measures  $\mu_{n}=F^{n} \mu_{0}$ obtained from the initial
measure by the $n$-fold action of the DS is asymptotically
equivalent, as $n\rightarrow\infty,$ to  a sequence of normal
measures with variances  $\sigma_{n}^{2}$ and  mean values
$\xi_{n}$
\end{definition}

The main problems connected with the deterministic diffusion for
DS is to establish the fact of existence of this  diffusion in
terms of the functions $f$ specifying the DS (\ref{eq:1}). The
construction of an efficient algorithm for the determination the
coefficient of deterministic diffusion $D$ and  drift $\xi_{n}$
seems to be an important problem, especially from the viewpoint of
applications. At present, a series of expressions is deduced and
various numerical methods for the analysis of the dependence of
the diffusion coefficient on the form of the functions $f$ are
developed. Note that the dependence of the coefficient of
deterministic diffusion for the LDS (\ref{eq:1})--(\ref{eq:3}) on
$\Lambda$ is quite complicated (nowhere differentiable fractal
dependence) even for the linear function $f(x)=\Lambda x$ in the
interval $I_0=[-\frac{1}{2},\frac{1}{2}),$ \cite{CAM,Kl}.

We now present heuristic arguments for the existence of
deterministic diffusion for the LDS with the linear function
$f(x)=\Lambda x,$\, $\Lambda>2,$ in the interval
$I_0=[-\frac{1}{2},\frac{1}{2}).$ In this case, the LDS
(\ref{eq:1})--(\ref{eq:3}) can be represented in the form

\begin{equation}\label{eq:13}
x_{n+1}=x_{n}+(\Lambda-1) \{x_{n}),
\end{equation}
where $\{x_{n})=x-[ x)$ is the fractional part of the number  $x$
and $[ x)$ is the nearest integer for the number $x.$ Equation
(\ref{eq:13}) can be represented in the equivalent form
\begin{equation}\label{eq:14}
x_{n+1}=x_{0}+\sum\limits_{k=1}^{n}(\Lambda-1) \{x_{k}).
\end{equation}

If the map $f $ is stretching, i.e., $\Lambda>2,$ and the initial
value $x_{0}$  takes values in the interval $I_0$ with a certain
(e.g., constant) probability density, then we can assume that the
quantities $\{x_{k}),$ the fractional parts of $x_{k}$ are
uniformly distributed over the interval $I_0$ and independent for
different $k.$ The rigorous substantiation of the uniformity of
distributions of the fractional parts for different stretching
maps can be found in \cite{Ka2}. According to (\ref{eq:14}), the
quantities  $x_{n}$ can be regarded as the sum of $x_{0}$ and $n$
independent identically distributed random variables. Thus, by the
central limit theorem, the quantities $x_{n},$ as
$n\rightarrow\infty,$ are distributed according to the normal law
with zero mean value and the variance equal to the sum of
variances of the terms. Since the variance of $\{x),$ regarded as
a variable uniformly distributed in the interval
$[-\frac{1}{2},\frac{1}{2}),$ is equal to $\sigma^{2} =
\int\limits_{-\frac{1}{2}}^{\frac{1}{2}}\,x^2\,dx=\frac{1}{12},$
we get $\sigma^{2}(x_{n+1})=\frac{(\Lambda -1)^{2}}{12}.$ This
yields  the approximate relation for the diffusions coefficients
$D=\frac{(\Lambda -1)^{2}}{24}$ for any linear map $f(x)=\Lambda
x$ in the LDS (\ref{eq:1})--(\ref{eq:3}).

Let us turn to the case where the Markov subpartition is
consistent with the LDS according to Definition~\ref{D:2}. In this
case, by $P_{k,j}(n),$  we denote the densities in each
subinterval $I_{k,j}$ of the interval $I_k,$\, $k \in Z ,$\,
$j=1,2,...,m,$ for $n$ iterations. If we consider the collection
of $P_{k,j}(n)$ $j=1,2,...,m,$ as the components of the vector
$P_{k}(n)=\col(P_{k,1}^{(n)},...,P_{k,m}^{(n)}),$ then we get an
analog of Eq.~(\ref{eq:4.1}), where the quantities $P_k(n)$ are
vectors, $p_k=\{p_{k,i,j}\}|_{i,j=1}^{m}$ is the transition matrix
for the Perron--Frobenius operator, and $p_{k,i,j}$  is the
density of measure in $I_{k,i}$ in the case of  single application
of the LDS to the measure with unit density on $I_{\theta,j}.$
This vector analog of Eq.~(\ref{eq:4.1}) is also well studied and
the solution $P_k(n)$ leads to the normal distribution in the
$m$-dimensional space \cite{F,B}.

\begin{theorem}\label{Th:4.1}
Assume that the Markov subpartition $\{I_{k,j}\}_{k\in Z, 1\leq j
\leq n }$ of the axis is consistent with the action of the  LDS
(\ref{eq:1})--(\ref{eq:3}) and that the function $f(x)$ is
stretching and maps the interval $I_0=[-\frac{1}{2},\frac{1}{2})$
into a finite interval $[a,b)$ of length greater than 2. Then,
after the $n$-fold action of the LDS, the initial measure
$\mu_{0}$ in $I_0$ with bounded density is transformed into a
measure $\mu_{n}$ asymptotically equivalent, as
$n\rightarrow\infty,$ to a normal measure with densities
$P_{k,j}(n)$ in the intervals $I_{k,j}$
\begin{equation}\label{eq:4.7}
P_{k,j}(n)=\frac{\alpha_{j}}{2\sqrt{\pi D
n}}e^{-\frac{(k-\xi_n)^2}{4D n}},\,\, j=1,2,...,m,
\end{equation}
where  $\xi_n$ is the drift, $D$ is the coefficient of
deterministic diffusion, and the parameters $\alpha_{j}>0,$\,\,
$j=1,2,...,m,$ specify the distributions of densities in the
subintervals $I_{k,j}$ of the intervals $I_{k},$\,$k\in Z.$
\end{theorem}

\begin{proof}
As already indicated, the vectors
$P_{k}(n)=\col(P_{k}(n,1),...,P_{k,m}(n))$ satisfy the equation

\begin{equation}\label{eq:4.8}
P_{k}(n+1)=\sum\,P_{j}P_{k-j}(n),
  \end{equation}
where the matrix $P_{j}$ is expressed via the translation matrices
in the considered LDS. Note that the matrix  $E=\sum_{j}\, P_{j}$
is equivalent to the stochastic irreducible matrix $DED^{-1},$
where is a $D$-diagonal matrix with  the lengths of subintervals
$I_{0,1}, ...,I_{0,m}$ on the diagonal. As a result of the Fourier
transformation, Eq.~(\ref{eq:4.8}) is transformed into the
difference equation
\begin{equation}\label{eq:4.9}
P(\lambda,n+1)=P(\lambda)P(\lambda,n),
  \end{equation}
where the matrix  $P(\lambda)=\sum\limits_{j}\,
P_{j}e^{ij\lambda}$ and the vector
$P(\lambda,n)=\sum\limits_{k}\,P_{k}(n)e^{ik \lambda}.$

The solution of Eq.~(\ref{eq:4.9}) is explicitly expressed via the
eigenvalues and eigenvectors (including the adjoined vectors in
the case of multiple eigenvalues) of the matrix $P(\lambda).$ If
$z(\lambda)$ is the maximum eigenvalue of the matrix $P(\lambda)$
(this eigenvalue is simple), then the solution of
Eq.~(\ref{eq:4.9}) as $n\rightarrow\infty$ can be represented in
the form
\begin{equation}\label{eq:4.10}
P(\lambda,n)=z^{n}(\lambda)\alpha(\lambda)+o(1),\,n\rightarrow\infty,
  \end{equation}
where  $\alpha(\lambda)$ is the eigenvector of the matrix
$P(\lambda)$ corresponding to the eigenvalue $z(\lambda).$ Note
that $z(0)=1$ and all components of the vector $\alpha(0)$ are
positive. If we perform the inverse Fourier transformation, then
we get relation (\ref{eq:4.7}) with
$D=-\frac{1}{2}\frac{\partial^{2}}{\partial
\lambda^{2}}z(\lambda)|_{\lambda=0},$\,\, $\xi=i\frac{\partial
z(\lambda)}{\partial \lambda}|_{\lambda=0}$ from (\ref{eq:4.10}).
\end{proof}

\begin{remark}
By using Theorem~\ref{Th:4.1}, one can deduce the explicit
``parametric''  dependence of  $D$ on $\Lambda$ for the linear
function $f(x)=\Lambda x$ specifying the LDS. In this case, the
role of parameters is played by the characteristics $(s,\vec{n})$
of the Markov subpartition $\{I_{k,j}\}$ consistent with the LDS,
that is by the parity of $m$ and the integer-valued vector
$\vec{n}.$ By using these parameters, we can explicitly construct
three polynomials $R_{\vec{n}}(x),$ \,$P_{\vec{n}}(x),$ and
$Q_{\vec{n}}(x)$ with integer-valued coefficients such that the
slope $\Lambda$ is the maximum root of the polynomial
$R_{\vec{n}}(x),$ i.e., $R_{\vec{n}}(x)=0$ (see
Proposition~\ref{p:3aa}) and $\displaystyle
D=\frac{P_{\vec{n}}(x)}{Q_{\vec{n}}(x)}.$
\end{remark}

\begin{example}
\end{example}

As an example, we consider the case of LDS with a linear function
$f(x)=\Lambda x,$\, where $\Lambda=2s$ is an even number. In this
case, the Markov subpartition $I_k$ is formed by two subintervals
$I_{k,+}=[k,k+\frac{1}{2})$ and $I_{k,-}=[k-\frac{1}{2},k).$ If
the first components of the vectors are referred to the intervals
$I_{k,+}$ and the second components are referred to the intervals
$I_{k,-},$ then $p_0=\displaystyle\frac{1}{\Lambda} \left(
\begin{array}{lr}
 1 & 0 \\
0 & 1
\end{array}
\right), $\, $p_j=\displaystyle\frac{1}{\Lambda} \left(
\begin{array}{lr}
 1 & 0 \\
1 & 0
\end{array}
\right), $\,\, $p_s=\left(
\begin{array}{lr}
 0 & 0 \\
1 & 0
\end{array}
\right), $\,\, $p_{-j}=\displaystyle\frac{1}{\Lambda} \left(
\begin{array}{lr}
 0 & 1 \\
0 & 1
\end{array}
\right), $\, and \,$p_{-s}=\displaystyle\frac{1}{\Lambda}\left(
\begin{array}{lr}
 0 & 1 \\
0 & 0
\end{array}
\right), $\,\,$j=1,2,...,s-1.$ In this case, the characteristic
function  $ P(\lambda,n)=\sum\limits_{k}\,P_{k}(n)e^{i\lambda k}$
is vector-valued and
$$
P(\lambda)=\sum_{j}\,p_je^{ij\lambda}=\frac{\sin\frac{\lambda
}{2}s}{2s\sin\frac{\lambda}{2}}\left(
\begin{array}{lr}
 e^{i\frac{\lambda}{2}(s-1)} & e^{-i\frac{\lambda}{2}(s+1)} \\
e^{i\frac{\lambda}{2}(s+1)} & e^{-i\frac{\lambda}{2}(s-1)}
\end{array}
\right).
$$
The determinant of the matrix $P(\lambda)$ is equal to zero and
the trace $\displaystyle \sp p(\lambda)=\frac{\sin\frac{\lambda
}{2}s}{s\sin\frac{\lambda}{2}}\cos \frac{\lambda}{2}(s-1).$ Hence,
the nontrivial eigenvalue $z(\lambda)$ of the matrix $P(\lambda)$
coincides with the trace of the matrix $P(\lambda).$ This yields
the following  well-known explicit expression \cite{CAM} for the
diffusion coefficients for even  $\Lambda=2s$:
\begin{equation}\label{eq:21}
D=-\frac{1}{2}\frac{\partial^{2}}{\partial\lambda^{2}} \sp
p(\lambda)|_{\lambda=0}=\frac{(\Lambda-1)(\Lambda-2)}{24}.
\end{equation}

As it follows from the relations of Example~\ref{Ex:5} and
(\ref{eq:21}) for the coefficient of deterministic diffusion in
the case of a linear function $f(x)=\Lambda x$ with integer
$\Lambda,$ $D(\Lambda)$ is a monotonic function of $\Lambda.$ As
$\Lambda$ increases, i.e., the degree of stretching of the map
$f(x)$ increases,  the coefficient of deterministic diffusion
$D(\Lambda)$ increases. However, if we compare the relations for
even and odd $\Lambda,$ then, e.g., we get $D_3=\frac{1}{3}$ for
$\Lambda=3$ and  $D_4=\frac{1}{4}$ for  $\Lambda=4.$ Thus, at
first sight, these results seem to be intuitively strange.

We consider this problem in more detail. Assume that the initial
probability measure is uniformly distributed over the interval $
[\frac{3}{2},2),$ that is its density is constant and equal to 2. As
a result of the one-time action of the LDS, we get a measure with
constant density in the interval $ [\frac{1}{2},2)$ for the
mapping with $\Lambda=3$  and, in the interval $[0,2)$ for the
mapping with $\Lambda=4.$ Clearly, the variance of a measure with
constant density in the interval $[0,2)$ is greater than the
variance of a measure with constant density in the interval $
[\frac{1}{2},2),$ which agrees with our assumption that the higher
the degree of stretching, the greater the variance. A similar
picture is observed in the case where the initial measure has a
constant density in the interval $ [-2, -\frac{3}{2}).$
 For $\Lambda=3,$ the mapping realized by the LDS leads to a measure with constant density in the interval $[-2,-\frac{1}{2})$ and the mapping with $\Lambda=4$ yields a measure with
constant density in the interval $ [-2,0).$

However, if the initial probability measure has a constant density
in the union of intervals $ [-2, -\frac{3}{2})\bigcup
[\frac{3}{2},2),$ then, in view of the linearity of the
Perron--Frobenius operator, for $\Lambda=3,$ we get a measure in
the union of intervals $A_3=[-2, -\frac{1}{2})\bigcup
[\frac{1}{2},2],$ whereas for the mapping with  $\Lambda=4,$ the
measure has a constant density in the interval $A_4=[-2,2].$
Clearly, the variance of the probability measure on $A_3$ is
greater than the variance on $A_4.$ Thus, the higher degree of
stretching not always corresponds to a greater variance.
\textbf{If stretching is directed toward the mean value, then the
variance decreases.}

The relations for the values of $D$ for even and odd $\Lambda$ can
be replaced by a single relation valid for all  $\Lambda$ by
introducing a function $\omega(\Lambda),$ namely,
\begin{equation}\label{eq:4.11}
D(\Lambda)=\frac{1}{24}(\Lambda-1)(\Lambda-\omega(\Lambda)),
\end{equation}
where the function $\omega(\Lambda)$ is equal to 2 for even
$\Lambda$ and to --1 for odd $\Lambda.$ For any $\Lambda,$ the
function $\omega(\Lambda)$ is characterized by a complex fractal
dependence on $\Lambda$ caused by the fractal dependence of $D$ on
$\Lambda$ (see \cite{CAM, Kl}). The function  $\omega(\Lambda)$
can be approximately regarded as a 2-periodic function such that
$\omega(\Lambda)=2-3|\Lambda-4|$ for $\Lambda \in [3,5].$ The plot
of this function is depicted in Fig.~1 and gives the first
approximation of the behavior of $D$ as a function of $\Lambda.$

\begin{figure}[h]
\centering
\includegraphics[width=0.8\textwidth]{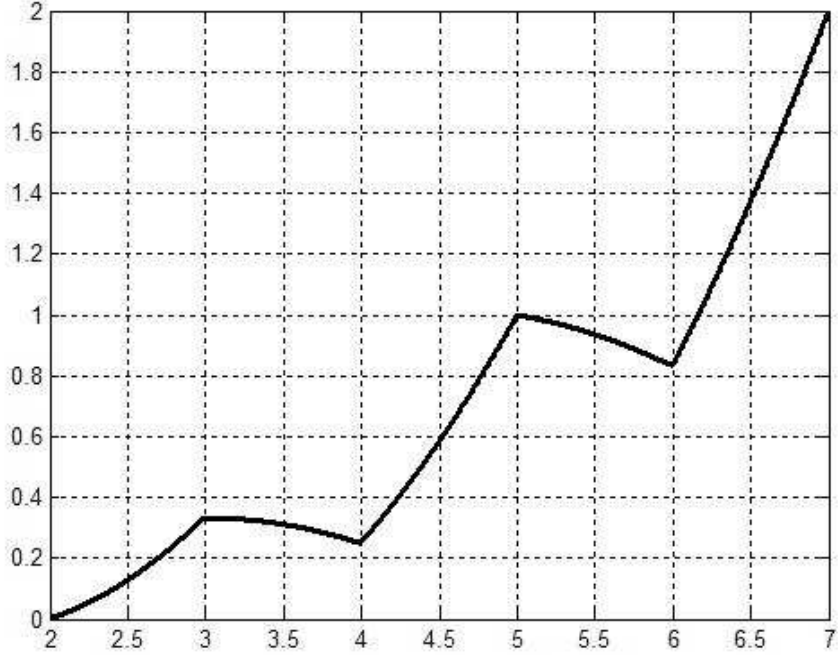}
\caption { Graph of  the function $D(\Lambda)$.}
\end{figure}

Let us consider some examples of a linear function $f(x)=\Lambda
x,$ that illustrated the results of theorem~\ref{Th:4.1}.
\begin{example}
1. Markov subpartition of $m=3$ subintervals:
$I_{k,1}=[k-\frac{1}{2},k-\xi),$\,$I_{k,2}=[k-\xi,k+\xi),$\,$I_{k,3}=[k+\xi,k+\frac{1}{2}),$\,$k\in Z.$\\
2. A system of two equations for $\xi$ and  the value of slope $\Lambda:$\\
$$
\Lambda\xi=\frac{1}{2}, \qquad  \Lambda\frac{1}{2}=2-\xi.
$$
3. An equation for $\Lambda:$\qquad $\Lambda^{2}-4\Lambda+1=0.$\\
4. A value of $\Lambda:$\,\,$\Lambda=2+\sqrt{3}\approx 3.73. $\\
5. A value of $\xi:$\,\,$\xi=\frac{2-\sqrt{3}}{2}\approx 0.13. $\\
6. A matrix $P(\lambda):$\\
$$P(\lambda)=\displaystyle \frac{1}{\Lambda}\left(
\begin{array}{ccc}
e^{-i\lambda}  & 1 & e^{i\lambda} +e^{2i\lambda}\\
e^{-i\lambda} & 1 &e^{i\lambda}\\
e^{-i\lambda} +e^{-2i\lambda} & 1 & e^{i\lambda}
\end{array}
\right).
$$
7. A value of greatest eigenvalue $z(\lambda)$  of matrix
$P(\lambda):$\\
$z(\lambda)=\frac{1}{\lambda}(1+\cos\lambda +\sqrt{\cos^{2}\lambda+2\cos\lambda}).$\\
8.  A value of deterministic diffusion coefficient:\,
$D=\frac{\sqrt{3}}{6}\approx 0.29.$\\
9.  Values of densities $\alpha_j$ on subintervals $I_{k,j}:$\,
$\alpha_1=\alpha_3=\frac{2\sqrt{3}+3}{6}\approx 1.08,$\,\,
$\alpha_2=\frac{3+\sqrt{3}}{6}\approx 0.79.$
\end{example}
\begin{example}
1. Markov subpartition of $m=3$ subintervals:\\
$I_{k,1}=[k-\frac{1}{2},k-\xi),$\,$I_{k,2}=[k-\xi,k+\xi),$\,$I_{k,3}=[k+\xi,k+\frac{1}{2}),$\,$k\in Z.$\\
2. A system of two equations for $\xi$ and  the value of slope $\Lambda:$\\
$$\Lambda\xi=\frac{3}{2},\,\,\,\qquad \Lambda\frac{1}{2}=3-\xi.
$$
3. An equation for $\Lambda:$\,\, $\Lambda^{2}-6\Lambda+3=0.$\\
4. A value of $\Lambda:$\,\,$\Lambda=3+\sqrt{6}\approx 5.45.$\\
 5. A value of $\xi:$\,\,$\xi=\frac{3-\sqrt{6}}{2}\approx 0.28.$\\
 6.  A matrix $P(\lambda):$\\ $$P(\lambda)=\displaystyle \frac{1}{\Lambda}\left(
\begin{array}{ccc}
e^{-2i\lambda}  & 1 + e^{i\lambda} +e^{-i\lambda} & e^{2i\lambda} +e^{3i\lambda}\\
e^{-2i\lambda} & 1 + e^{i\lambda} +e^{-i\lambda} & e^{2i\lambda}\\
e^{-2i\lambda} +e^{-3i\lambda} & 1+ e^{i\lambda} +e^{-i\lambda} &
e^{2i\lambda}
\end{array}
\right).
$$
7. An equation for greatest eigenvalue $z(\lambda)$ of matrix $P(\lambda):$\\
$\Lambda^{3}z^{3} -\Lambda^{2}z^{2}(1+2\cos\lambda +2\cos2\lambda)-\Lambda z (1+2\cos\lambda)+1+2\cos\lambda=0.$\\
8. A value of deterministic diffusion coefficient:\\
$ \displaystyle D=\frac{31 \Lambda -16}{36\Lambda -24}\approx 0.89.$ \\
9.  Values of densities $\alpha_j$ on subintervals $I_{k,j}:$\,
$\alpha_1=\alpha_3=\frac{\sqrt{6}+2}{2}\approx 2.22,$\,\,
$\alpha_2=\frac{3+\sqrt{6}}{6}\approx 0.9.$
\end{example}
\begin{example}
1. Markov subpartition of $m=3$ subintervals:\\
$I_{k,1}=[k-\frac{1}{2},k-\xi),$\,$I_{k,2}=[k-\xi,k+\xi),$\,$I_{k,3}=[k+\xi,k+\frac{1}{2}),$\,$k\in
Z.$\\  2. A system of two equations for $\xi$ and  the value of slope $\Lambda:$\\
$$\Lambda\xi=\frac{3}{2},\,\,\,\qquad \Lambda\frac{1}{2}=2+\xi.
$$
3.  An equation for $\Lambda:$\,\, $\Lambda^{2}-4\Lambda-3=0.$\\
4. A value of $\Lambda:$\,\, $\Lambda=2+\sqrt{7}\approx 4.65.$\\
 5. A value of  $\xi:$\,\, $\xi=\frac{\sqrt{7}-2}{2}\approx 0.32.$\\
 6.  A matrix $P(\lambda):$\\$$P(\lambda)=\displaystyle \frac{1}{\Lambda}\left(
\begin{array}{ccc}
0  & 1 + e^{i\lambda} +e^{-i\lambda} & e^{2i\lambda} \\
e^{-2i\lambda} & 1 + e^{i\lambda} +e^{-i\lambda} & e^{2i\lambda}\\
e^{-2i\lambda}  & 1+ e^{i\lambda} +e^{-i\lambda} & 0
\end{array}
\right).
$$
7. An equation for greatest eigenvalue $z(\lambda)$ of matrix $P(\lambda):$\\
$-\Lambda^{3}z^{3} +\Lambda^{2}z^{2}(1+2\cos\lambda )+\Lambda z (1+2\cos2\lambda(1+2\cos\lambda))+1+2\cos\lambda=0.$\\
8. A value of deterministic diffusion coefficient:\\
$\displaystyle D=\frac{9 \Lambda +2}{13\Lambda -9}=\frac{479+107\sqrt{7}}{894}\approx 0.85.$ \\
9.  Values of densities $\alpha_j$ on subintervals $I_{k,j}:$\,
$\alpha_1=\alpha_3=\frac{1}{2}+\frac{1}{\sqrt{7}},$\,\,
$\alpha_2=\frac{1}{2}+\frac{1}{2\sqrt{7}}.$
\end{example}
\begin{example}
1. Markov subpartition of $m=4$ subintervals:\\
$I_{k,1}=[k-\frac{1}{2},k-\xi),$\,$I_{k,2}=[k-\xi,k),$\,$I_{k,3}=[k,k+\xi),$\,$I_{k,4}=[k+\xi,k+\frac{1}{2}),$\,$k\in Z.$ \\
2. A system of two equations for $\xi$ and  the value of slope $\Lambda:$\\
$$\Lambda\xi=1,\,\,\,\qquad \Lambda\frac{1}{2}=1+\xi.
$$
3.  An equation for $\Lambda:$\,\,\qquad  $\Lambda^{2}-2\Lambda-2=0.$\\
4. A value of $\Lambda:$\,\, $\Lambda=1+\sqrt{3}\approx 2.73.$\\
 5. A value of  $\xi:$\,\,$\xi=\frac{\sqrt{3}-1}{2}\approx 0.37.$\\
 6.  A matrix $P(\lambda):$\\$$P(\lambda)=\displaystyle \frac{1}{\Lambda}\left(
\begin{array}{cccc}
0  & 1 & e^{i\lambda} & 0 \\
e^{-i\lambda} & 1 & e^{i\lambda} & 0\\
0  &  e^{-i\lambda} &1&e^{i\lambda}\\
0  &  e^{-i\lambda} &1&0
\end{array}
\right).
$$
7.  A value of greatest eigenvalue $z(\lambda)$  of matrix
$P(\lambda):$\\
$  z(\lambda)=\frac{1}{\Lambda}(1+\sqrt{1+2\cos \lambda}).$\\
8. A value of deterministic diffusion coefficient:\,
$D=\frac{3-\sqrt{3}}{12}\approx 0.11.$ \\
9.  Values of densities $\alpha_j$ on subintervals $I_{k,j}:$\,
\,$\alpha_1=\alpha_4=\frac{3+\sqrt{3}}{6}\approx 0.79,$\,\,
$\alpha_2=\alpha_{3}=\frac{3+2\sqrt{3}}{6}\approx 1.08.$
\end{example}
\begin{example}
1. Markov subpartition of $m=4$ subintervals:\\
$I_{k,1}=[k-\frac{1}{2},k-\xi),$\,$I_{k,2}=[k-\xi,k),$\,$I_{k,3}=[k,k+\xi),$\,$I_{k,4}=[k+\xi,k+\frac{1}{2}),$\,$k\in Z.$ \\
 2. A system of two equations for $\xi$ and  the value of slope $\Lambda:$\\
$$\Lambda\xi=1,\,\,\,\qquad \Lambda\frac{1}{2}=2-\xi.
$$
3.  An equation for $\Lambda:$\,\,\qquad  $\Lambda^{2}-4\Lambda+2=0.$\\
4. A value of $\Lambda:$\,\, $\Lambda=2+\sqrt{2}\approx 3.41.$\\
 5. A value of $\xi:$\,\,$\xi=\frac{2-\sqrt{2}}{2}\approx 0.29.$\\
 6.  A matrix $P(\lambda):$\\$$P(\lambda)=\displaystyle \frac{1}{\Lambda}\left(
\begin{array}{cccc}
e^{-i\lambda}   & 1 & e^{i\lambda} & e^{2i\lambda}  \\
e^{-i\lambda} & 1 & e^{i\lambda} & 0\\
0  &  e^{-i\lambda} &1&e^{i\lambda}\\
e^{-i\lambda}   &  e^{-i\lambda} &1&e^{i\lambda}
\end{array}
\right).
$$
7. A value of  greatest eigenvalue $z(\lambda)$ of matrix $P(\lambda):$\\
$  z(\lambda)=\frac{1}{\Lambda}(1+\cos \lambda+\sqrt{2-\sin^{2}\lambda}).$\\
8. A value of deterministic diffusion coefficient:
\,$D=\frac{1}{4}.$ \\
9.  Values of densities $\alpha_j$ on subintervals $I_{k,j}:$\,
$\alpha_1=\alpha_4=\frac{\sqrt{2}-1}{2}\approx 0.21,$\,\,
$\alpha_2=\alpha_{3}=\frac{2-\sqrt{2}}{2}\approx 0.29.$
\end{example}
\begin{example}
1. Markov subpartition of $m=3$ subintervals:\\
$I_{k,1}=[k-\frac{1}{2},k-\xi_{2}),$\,$I_{k,2}=[k-\xi_{2},k-\xi_{1}),$\,$I_{k,3}=[k-\xi_{1},k+\xi_{1}),$\,$I_{k,4}=[k+\xi_{1},k+\xi_{2}),$\,$I_{k,5}=[k+\xi_{2},k+\frac{1}{2}),$\,$k\in Z.$ \\
2. A system of  equations for $\xi$ and  the value of slope $\Lambda:$\\
$$\Lambda\xi_{1}=\frac{3}{2},\,\,\,\qquad \Lambda\xi_{2}=2-\xi_{1},\,\,\,\qquad \Lambda\frac{1}{2}=2+\xi_{2}.
$$
3.  An equation for $\Lambda:$\,\, \qquad  $\Lambda^{3}-4\Lambda^{2}-4\Lambda+3=0.$\\
4. A value of $\Lambda:$\,\,$\Lambda\approx 4.75.$\\
 5.  A value of  $\xi:$\,\,$\xi_{1}=\frac{3}{2\Lambda},\,\,\,\xi_{2}=\frac{4\Lambda-3}{2\Lambda^{2}}.$\\
 6.  A matrix $P(\lambda):$\\ $$P(\lambda)=\displaystyle \frac{1}{\Lambda}\left(
\begin{array}{ccccc}
0  & 0&             1 + e^{i\lambda} + e^{-i\lambda}& e^{2i\lambda}& 0 \\
e^{-2i\lambda} &0 & 1 + e^{i\lambda}+ e^{-i\lambda}& e^{2i\lambda}  & 0\\
e^{-2i\lambda} &0 & 1 + e^{i\lambda}+ e^{-i\lambda} & 0 & e^{2i\lambda} \\
0 &      e^{-2i\lambda} & 1 + e^{i\lambda}+ e^{-i\lambda} & 0 & e^{2i\lambda} \\
0 &      e^{-2i\lambda} & 1 + e^{i\lambda}+ e^{-i\lambda} &0 & 0
\end{array}
\right).
$$
7. An equation for greatest eigenvalue $z(\lambda)$ of matrix $P(\lambda):$\\
$\Lambda^{5}z^{5}-\Lambda^{4}z^{4}(1+2\cos \lambda)-\Lambda^{3}z^{3}(1+2\cos \lambda(1+2\cos \lambda))-\Lambda^{2}z^{2}(1+2\cos \lambda+2\cos 2\lambda)-z+1+2\cos \lambda=0.$\\
8. A value of deterministic diffusion coefficient:
\,$\displaystyle D=\frac{81\Lambda^{2}+69\Lambda-55}{131\Lambda^{2}+99\Lambda-93}\approx 0.64.$ \\
9.  Values of densities $\alpha_j$ on subintervals $I_{k,j}:$\,
$\alpha_1=\alpha_5=\frac{\Lambda^{2}+4\Lambda-3}{4\Lambda(9-\Lambda)}
,$\,\,
$\alpha_2=\alpha_{4}=\frac{\Lambda^{2}+4\Lambda-3}{16(9-\Lambda)},$\,\,
$\alpha_3=\frac{5\Lambda^{2}-11\Lambda+6}{16(9-\Lambda)}.$
\end{example}
\begin{example}
1. Markov subpartition of $m=6$ subintervals:\\
$I_{k,1}=[k-\frac{1}{2},k-\xi_{2}),$\,$I_{k,2}=[k-\xi_{2},k-\xi_{1}),$\,$I_{k,3}=[k-\xi_{1},k),$\,$I_{k,4}=[k,k+\xi_{1}),$\,$I_{k,5}=[k+\xi_{1},k+\xi_{2}),$\,$I_{k,6}=[k+\xi_{2},k+\frac{1}{2}),$\,$k\in Z.$\\
2. A system of equations for $\xi$ and  the value of slope $\Lambda:$\\
$$\Lambda\xi_{1}=\xi_{2},\,\,\,\qquad \Lambda\xi_{2}=2-\xi_{1},\,\,\,\qquad \Lambda\frac{1}{2}=2+\xi_{1}.
$$
3.  An equation for $\Lambda:$\,\,\qquad  $\Lambda^{3}-4\Lambda^{2}+\Lambda-8=0.$\\
4. A value $\Lambda:$\,\,$\Lambda\approx 4.22.$\\
 5.  A value $\xi:$\,\,$\xi_{1}=\frac{2}{\Lambda^{2}+1},\,\,\,\xi_{2}=\frac{2\Lambda}{\Lambda^{2}+1}.$\\
 6.  A matrix $P(\lambda):$\\$$P(\lambda)=\displaystyle \frac{1}{\Lambda}\left(
\begin{array}{cccccc}
0  &   1 + e^{-i\lambda} & 0 &   0 &    e^{i\lambda} + e^{2i\lambda} & 0 \\
0 &  e^{-i\lambda} & 1 &0& e^{i\lambda}+ e^{2i\lambda} & 0\\
e^{-2i\lambda} &  e^{-i\lambda}&1 &0 &  e^{i\lambda} & e^{2i\lambda} \\
e^{-2i\lambda} &  e^{-i\lambda}& 0 & 1 &  e^{i\lambda} & e^{2i\lambda} \\
0 &  e^{-i\lambda}+ e^{-2i\lambda}&0 & 1 &e^{i\lambda} & 0 \\
0 &  e^{-i\lambda}+ e^{-2i\lambda}&0 & 0 &1+e^{i\lambda} & 0
\end{array}
\right).
$$
7. An equation for greatest eigenvalue $z(\lambda)$ of matrix $P(\lambda):$\\
$ -\Lambda^{4}z^{4}+\Lambda^{3}z^{3}(1+2\cos \lambda)+\Lambda^{2}z^{2}(1+2\cos \lambda)+z(1+2\cos \lambda+2\cos 2\lambda+2\cos 3\lambda)+2+2\cos 2 \lambda+4\cos \lambda\cos 2 \lambda=0.$\\
8. A value of deterministic diffusion coefficient:\,
$\displaystyle D=\frac{5\Lambda^{2}+26\Lambda+22}{18\Lambda^{2}+18\Lambda+56}\approx 0.49.$ \\
9.  Values of densities $\alpha_j$ on subintervals $I_{k,j}:$\,
$\alpha_1=\alpha_6=\frac{\Lambda^{2}+1}{3\Lambda^{2}-8\Lambda+1},$\,\,
$\alpha_2=\alpha_{5}=\frac{\Lambda^{2}+2}{3\Lambda^{2}-8\Lambda+1},$\,\,
$\alpha_3=\alpha_4=\frac{\Lambda^{2}+\Lambda+5}{3\Lambda^{2}-8\Lambda+1}.$
\end{example}
\begin{example}
1. Markov subpartition of $m=7$ subintervals:\\
$I_{k,1}=[k-\frac{1}{2},k-\xi_{3}),$\,$I_{k,2}=[k-\xi_{3},k-\xi_{2}),$\,
$I_{k,3}=[k-\xi_{2},k-\xi_{1}),$\,$I_{k,4}=[k-\xi_{1},k+\xi_{1}),$\,$I_{k,5}=[k+\xi_{1},k+\xi_{2}),$\,$I_{k,6}=[k+\xi_{2},k+\xi_{3})$,\,$I_{k,7}=[k+\xi_{3},k+\frac{1}{2})$\,$k\in Z.$ \\
2. A system of  equations for $\xi$ and  the value of slope $\Lambda:$\\
$$\Lambda\xi_{1}=\xi_{2},\qquad \Lambda\xi_{2}=\xi_{3},\qquad \Lambda\xi_{3}=\frac{1}{2},\qquad \Lambda\frac{1}{2}=2-\xi_{1}.
$$
3. An equation for $\Lambda:$\,\,\qquad $\Lambda^{4}-4\Lambda^{3}+1=0.$\\
4. A value of $\Lambda:$\,\,$\Lambda\approx 3.968.$\\
 5. A value of $\xi:$\,\,$\xi_{1}=\frac{1}{2\Lambda^{3}},\,\,\,\xi_{2}=\frac{1}{2\Lambda^{2}},\,\,\,\xi_{3}=\frac{1}{2\Lambda}.$\\
 6.  A matrix $P(\lambda):$\\
 $$P(\lambda)=\displaystyle \frac{1}{\Lambda}\left(
\begin{array}{ccccccc}
e^{-i\lambda} &  1 & 0 &   0 & 0 &   0 &    e^{i\lambda} + e^{2i\lambda} \\
e^{-i\lambda} &  0 & 1 &   0 & 0 &   0 &    e^{i\lambda} + e^{2i\lambda} \\
e^{-i\lambda} &  0 & 0 &   1 & 0 &   0 &    e^{i\lambda} + e^{2i\lambda} \\
e^{-i\lambda} &  0 & 0 &   1 & 0 &   0 &    e^{i\lambda}  \\
e^{-i\lambda}+ e^{-2i\lambda}  &  0 & 0 &   1 & 0 &   0 &    e^{i\lambda}  \\
e^{-i\lambda}+ e^{-2i\lambda}  &  0 & 0 &   0 & 1 &   0 &    e^{i\lambda}  \\
e^{-i\lambda}+ e^{-2i\lambda}  &  0 & 0 & 0   & 0 &   1 &
e^{i\lambda}
\end{array}
\right).
$$
7. An equation for greatest eigenvalue $z(\lambda)$ of matrix $P(\lambda):$\\
$ \Lambda^{4}z^{4}-\Lambda^{3}z^{3}(2+2\cos \lambda)+ 1=0.$\\
8. A value of deterministic diffusion coefficient:\,
$D=\frac{1}{4(\Lambda-3)}\approx 0.26.$ \\
9.  Values of densities $\alpha_j$ on subintervals $I_{k,j}:$\,
$\alpha_1=\alpha_7=\frac{\Lambda}{4(\Lambda-3)} ,$\,\,
$\alpha_2=\alpha_{6}=\frac{\Lambda}{4},$\,\,
$\alpha_3=\alpha_5=\frac{(\Lambda^{2}-3\Lambda-3)\Lambda}{4(\Lambda-3)},$\,\,
$\alpha_4=\frac{\Lambda^{2}}{4}-\frac{3}{4}-\frac{5}{2(\Lambda-3)}.$
\end{example}

The results of theorem~\ref{Th:4.1} and above-considered examples
allow us to give the following definition of deterministic
diffusion of LDS~(\ref{eq:1})-(\ref{eq:3}).

\begin{definition}\label{D:6}
We say that the LDS (\ref{eq:1})--(\ref{eq:3})  has a
deterministic diffusion if, for any initial probability measure
$\mu_{0}$  with bounded density, there exist a sequence of numbers
$\sigma_{n}^{2}> 0$ and  $\xi_{n}$ and 1--periodic function
$\alpha(x)\geq 0,$\,$\int\limits_{-1/2}^{1/2}\,\alpha(x)\,dx=1,$
such that the sequence of measures $\mu_{n}=F^{n} \mu_{0}$,
obtained from the initial measure by the $n$-fold action of the
 LDS, is asymptotically equivalent, as  $n\rightarrow\infty$  to  a sequence of measures
with densities
$$\rho_n(x)=\frac{\alpha(x)}{\sigma_n\sqrt{2\pi}}e^{-\frac{(x-\xi_n)^2}{2\sigma_n^2}}.$$
\end{definition}
\begin{remark}
The LDS (\ref{eq:1})--(\ref{eq:3}) on the entire axis generates
the associated DS
 $x_{n+1}=\hat{f}(x_n)$ on the finite segment
$[-\frac{1}{2},\frac{1}{2}),$ where the value of the function
$\hat{f}(x_n)$ is equal to the fractional part of $f(x),$ i.e.,
$\hat{f}(x_n)=\{f(x)),$ and the function $\hat{f}$ maps the
segment $[-\frac{1}{2},\frac{1}{2})$ into itself. We denote this
DS, which is a compactification of the LDS, by CLDS. The function
$\alpha(x)\geq 0$ in Definition~\ref{D:6} is the density of the
invariant probability measure for the CLDS.
\end{remark}


\section{Deterministic Diffusion in the  Case of Transportation in a Billiard Channel}\label{Sec:5}

The results of numerical experiments carried out in \cite{ARV}
demonstrate that the deterministic diffusion in long billiard
channels is anomalous. There are different theoretical models of
this type of diffusion transport in long channels. One of these
models can be found in \cite{AGNN}. The essence of this model can
be described as follows:

Consider a long billiard channel with boundaries in the form of
periodically repeated arcs slightly distorting the straight lines
of boundaries. Assume that the upper boundary of the channel is
symmetric to the lower boundary. Then the trajectory of motion of
a billiard ball regarded as a material point obeying the ideal law
of reflections from both boundaries can be studied in a channel of
half width with symmetric reflections of the trajectory in the
upper part of the channel into its lower part relative to the
straight middle line of the channel. Hence, the reflections from
this middle line can be regarded as ideal.

We approximately assume that the reflection from the lower part of
the distorted boundary is realized in its linear approximation.
However, the normal $\nu$ to the surface of actual reflection does
not coincide with normal to the rectilinear approximation of the
channel and is described by a known function periodic along the
axis of the channel. Assume that $x_{n-1},$ \,$x_{n},$ and
$x_{n+1}$ are the abscissas of points of three consecutive
reflections of the billiard ball and that the vector $\nu(x_{n})$
makes an angle $\alpha(x_{n})$ with the normal to the surface (see
Fig.~2). Then the analysis of the ideal reflection at the point
$x_{n}$ leads to the equality of the angle of incidence
$\varphi_{n}$ and the angle of reflection $\psi_{n}$ relative to
the vector $\nu.$ It is clear that
$\frac{x_{n}-x_{n-1}}{h}=\tan(\varphi-\alpha)$ and
$\frac{x_{n+1}-x_{n}}{h}=\tan(\psi+\alpha).$ Hence,
\begin{equation}\label{eq:16}
\frac{x_{n+1}-x_{n}}{h}=\frac{\frac{x_{n+1}-x_{n}}{h}+\tan(2\alpha)}{1-\tan(2\alpha)\cdot\frac{x_{n+1}-x_{n}}{h}}.
\end{equation}

For large $h,$ we get the following approximate relation from
(\ref{eq:16}):
$$
x_{n+1}-x_{n}=x_{n}-x_{n-1}+h\tan(2\alpha (x_{n})).
$$
\begin{figure}[h]
\centering
\begin{tikzpicture}[>=latex,scale=1.5]
\draw[dash pattern=on 0.7cm off 0.3cm] (0,2) -- (8,2);
\draw[name path=low-line] (0,0) -- (8,0);
\draw [<->] (0.5,0) -- (0.5,2) node[left,midway]{$\frac{h}{2}$};
\draw[double] ($(2.8,0)!1.2cm!(5,2)$) arc (42:77:1.2) node[right=0.5cm]{{\small $\psi$}};
\draw[double] ($(2.8,0)!1.3cm!(2.1,2)$) arc (109.5:77:1.3); 
\draw (2.7,1.5) node{{\small $\varphi$}};
\draw ($(2.8,0)!1cm!(2.8,2)$) arc (90:77:1); 
\draw (2.9,1.1) node{{\small $\alpha$}};
\draw[decoration={markings, mark=at position 0.7 with {\arrow{latex}}},postaction={decorate}] (0.8,0.4) -- (1.5,0) 
node[below] {{\small $x_{n-1}$}};
\draw[decoration={markings, mark=at position 0.75 with {\arrow{latex}}},postaction={decorate}] (1.5,0) -- (2.1,2);
\draw[decoration={markings, mark=at position 0.3 with {\arrow{latex}}},postaction={decorate}] (2.1,2) -- (2.8,0)
node[below] {{\small $x_{n}$}};
\draw[decoration={markings, mark=at position 0.75 with {\arrow{latex}}},postaction={decorate}] (2.8,0) -- (5,2);
\draw[decoration={markings, mark=at position 0.3 with {\arrow{latex}}},postaction={decorate}] (5,2) -- (7.2,0)
node[below] {{\small $x_{n+1}$}};
\draw[decoration={markings, mark=at position 0.75 with {\arrow{latex}}},postaction={decorate}] (7.2,0) -- (7.4,1);

\draw (0,0) .. controls (0.5,-0.25) and (1,0.23) .. (1.2,0); 
\draw (1.2,0) .. controls (1.4,-0.23) and (1.7,0.31) .. (2,0);
\draw (2,0) .. controls (2.1,-0.1) and (2.2,-0.15) .. (2.3,0);
\draw (2.3,0) .. controls (2.5,0.3) and (2.6,0.1) .. (2.8,0);
\draw (2.8,0) .. controls (3,-0.1) and (3.1,-0.2) .. (3.3,0);
\draw (3.3,0) .. controls (3.6,0.3) and (4.2,-0.3) .. (5,0);
\draw (5,0) .. controls (6,0.25) and (6.5,-0.2) .. (7.2,0);
\draw (7.2,0) .. controls (7.7,0.23) .. (8,0);

\draw[dash pattern=on 0.3cm off 0.3cm] (2.1,2) -- (2.1,0);
\draw (2.8,2) -- (2.8,0);
\draw[name path=nu,->] (2.8,0) --+ (77:1.6cm)node[anchor=west]{$\nu$};
\end{tikzpicture}
\caption{ Reflections of the billiard ball in a channel of
half width.}
\end{figure}
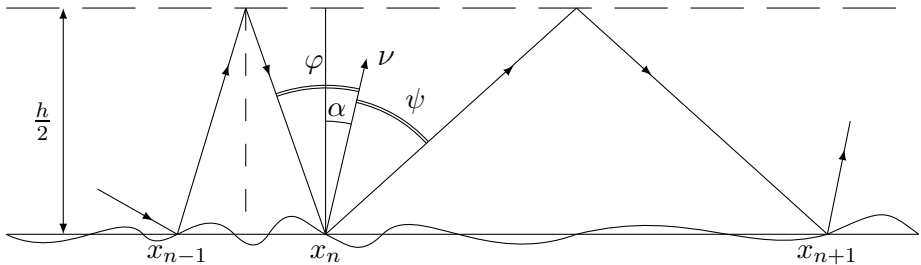

Let the function $f(x)=\tan(2\alpha x)$ be periodic in $x$ with
period 1. Then we arrive at the following model of a trajectory in
the billiard channel:
\begin{equation}\label{eq:17}
x_{n+1}-x_{n}=x_{n}-x_{n-1}+f(x_{n}).
\end{equation}
The dynamical system (\ref{eq:17}) is a two-dimensional dynamical
system discrete in time which generalizes the dynamical system
(\ref{eq:1}). The initial conditions for $x_{0}$ and $x_{1}$ are
assumed to be given. Let $x_{0}=0$ and let $x_{1}\in
I_{k}=[k-\frac{1}{2},k+\frac{1}{2}).$ Equation~(\ref{eq:17}) can
be represented in the equivalent form as follows:
\begin{equation}\label{eq:18}
x_{n+1}=x_{1}\cdot n + \sum\limits_{k=1}^{n}\,(n+1-k)f(x_{k}).
\end{equation}

Assume that the function $f(x)=\Lambda\cdot \{x)$ in
Eq.~(\ref{eq:18}) is linear with respect to the fractional part of
the argument. Then, for $\Lambda> 1$ and $x_{1}$ uniformly
distributed over the interval $I_{k},$  the terms in (\ref{eq:18})
can be regarded uniformly distributed independent random
variables. This leads to the normal distribution of $x_{n+1}$ with
variance equal to the sum of variances of all terms in
(\ref{eq:18}).

Hence, the variance of the distribution of $x_{n+1}$ is equal to
$\displaystyle \sigma^{2}=
\frac{1}{12}n^{2}+\frac{\Lambda^{2}}{12} \frac{n(n+1)(2n+1)}{6}$
and nonlinearly depends on $n\rightarrow \infty.$ The
deterministic diffusion in this billiard strip is anomalous. Thus,
there are two factors leading to the anomaly of the deterministic
diffusion, namely, the quadratic dependence of the variance on
distribution of the initial values even for the ideal billiard and
the growth of coefficients of the terms in sum (\ref{eq:18})
leading to the cubic dependence of the variance of $x_{n+1}$ on
$n.$

We now make an important remark. We have studied the distribution
of positions of the billiard ball after $n\rightarrow \infty$
reflections. From the physical point of view, it is more important
to get the distribution of the abscissas of billiard balls for
large values of time $t$ because, for the the same period of time,
the billiard ball makes different numbers of reflections for
different trajectories.

In the ordinary billiards, the time between two successive
collisions is proportional to the covered distance. However, in
the model of bouncing ball \cite{MK} over an irregular surface,
the time between two consecutive collisions is maximal if the
points of consecutive reflections coincide. The problem of
investigation of the anomalous deterministic diffusion in billiard
channels as $t\rightarrow\infty$ is of significant independent
interest.

Note that the Gaussian density in the case of ordinary diffusion
is the Green function of the Cauchy problem for the partial
differential equation
$$
\frac{\partial u}{\partial t}-D\Delta u=0.
$$

In the investigation of anomalous diffusion, we consider a
differential equation with fractional derivatives. Equations of
this kind are studied in numerous works (see \cite{K} and the
references therein).
\bigskip

\noindent  \textbf{Acknowledgments}

The authors express their gratitude to Prof.~A.~Katok for a
remarkable course of lectures on the contemporary theory of
dynamical systems delivered in May, 2014 in Kiev and for the
informal discussions which stimulated the authors to write this
paper.


\begin{thebibliography}{99}
\bibitem{AGNN}S. Albeverio, G. Galperin, I. Nizhnik and L. Nizhnik:
{\it Generalized billiards inside an infinite strip with periodic
laws of reflection along the strip's boundaries,} (2004), Preprint
Nr 04-08-154, pp. 28, BiBoS, Universit\"at Bielefeld, Regular and
Chaotic Dynamics {\bf 10} (3) (2005), 285-306.

\bibitem{ARV} D. Alonso, A. Ruiz, I. de Vega: {\it Transport in polygonal
billiards}, Physica D, 187 (2004), pp. 184-199.

\bibitem{B} R. N. Bhattacharya, R. Ranga Rao: Normal Approximation
and Asymptotic Expansions, 1976.

\bibitem{CAM}P. Cvitanovi{\'c}, R. Artuso, R. Mainieri, G.Tanner,
G. Vattay, N. Whelan, A. Wirzba: {Classical and Quantum Chaos}
(2004), http://www.nbi.dk/ChaosBook/

\bibitem{F} W. Feller: An Introduction to Probability Theory and
its Applications, 1968.

\bibitem {GK} B.V.~Gnedenko, A.N.~Kolmogorov:{ \it Limit distributions for
sums of independent random variables},  Addison-Wesley Mathematics
Series, Cambridge: Addison-Wesley Publishing Company, IX, 1954.
\bibitem {Ka2} B. Hasselblatt, A.Katok: A first course in dynamics
with a panorama of recent developments, Cambridge University
press, 2003.
\bibitem {Ka} A. Katok, B. Hasselblatt: Introduction ot the Modern
Theory of Dynamical Systems, Cambridge University press,
Cambridge, 1995.


\bibitem {Kl}  R.~Klages: Microscopic Chaos, Fractals and
Transport in Nonequilibrium Statistical Mechanics, World
Scientific, 2007.
\bibitem {Kl2}  R.~Klages, G.Radons, I.Sokolov (Eds): Anomalous transport,
Wiley-VCH, 2008.
\bibitem{K} A. N. Kochubei: Cauchy problem for fractional diffusion-wave
equations with variable coefficients // 2013, arXiv: 1308.6452v1,
31 p., http://arxiv.org/pdf/1308.6452v1.pdf, Applicable Analysis,
vol 93, issue 10, 2211-2242, 2014.

\bibitem {KK} N.~Korabel, R.~Klages: {\it Fractality of deterministic diffusion in
the nonhyperbolic climbing sine map,} CHAOTRAN proceedings in
Physica D 187, (2004),  66--88.


\bibitem{MK} L. Matyas, R. Klages: {\it Irregular diffusion in the bouncing
ball billiard}, Physica D, 187 (2004), pp. 165-183.

\bibitem {ShK}
A.N. Sharkovsky,  S.F. Kolyada, A.G. Sivak, V.V. Fedorenko,
Dynamics of one-dimensional maps [Translated from Russian],
Mathematics and its Applications, vol. 407, Kluwer Academic Publ.,
Dordrecht, 1997.

\end{thebibliography}
\end{document}